\documentclass[aps,tightenlines,10pt,superscriptaddress,onecolumn]{revtex4}%
\usepackage{amssymb}
\usepackage{amsfonts}
\usepackage{amsmath}
\usepackage{amsmath}
\usepackage{graphicx}%
\providecommand{\U}[1]{\protect\rule{.1in}{.1in}}
\newtheorem{theorem}{Theorem}

\newtheorem{definition}[theorem]{Definition}
\newtheorem{example}[theorem]{Example}

\newtheorem{lemma}[theorem]{Lemma}

\newtheorem{problem}[theorem]{Problem}
\newtheorem{proposition}[theorem]{Proposition}
\newtheorem{remark}[theorem]{Remark}

\newenvironment{proof}[1][Proof]{\noindent\textbf{#1.} }{\ \rule{0.5em}{0.5em}}
\begin{document}
\title{$\alpha$-Kuramoto partitions: graph partitions from the frustrated Kuramoto
model generalise equitable partitions\\ }
\author{Stephen Kirkland}
\affiliation{Hamilton Institute, National University of Ireland Maynooth, Maynooth, Co.
Kildare, Ireland.}
\author{Simone Severini}
\affiliation{Department of Computer Science, and Department of Physics \& Astronomy,
University College London, WC1E 6BT London, United Kingdom}

\begin{abstract}
The Kuramoto model describes the collective dynamics of a system of coupled
oscillators. An $\alpha$\emph{-Kuramoto partition} is a graph partition
induced by the Kuramoto model, when the oscillators include a phase
frustration parameter. We prove that every equitable partition is an $\alpha
$-Kuramoto partition, but that the converse does necessarily not hold. We give
an exact characterisation of $\alpha$-Kuramoto bipartitions.

\end{abstract}
\maketitle

\section{Introduction}

The Kuramoto model is a mathematical model of collective dynamics in a system
of coupled oscillators \cite{winfree67, kuramoto75}. It has been applied in
many contexts to describe synchronisation phenomena: \emph{e.g.}, in
engineering, for superconducting Josephson junctions, and in biology, for
congregations of synchronously flashing fireflies\emph{\ }\cite{acebron05};
the model has also been proposed to simulate neuronal synchronisation in
vision, memory, and other phenomena of the brain \cite{cumin2007}.

We consider a twofold generalisation of the Kuramoto model:\ firstly, while in
the standard case every oscillator is coupled with all the others, we
associate the oscillators to the vertices of a graph, as in \cite{strogatz01};
secondly, with the purpose of including an effect of the graph structure on
the dynamics, we take into account a phase frustration parameter, as in
\cite{nicosia12}.

Let $G=(V,E)$ be a connected graph on $n$ vertices without self-loops or
multiple edges. Let $A(G)$ be the adjacency matrix of $G$. For each vertex
$i\in V(G)$, we shall study the following equation, which defines a
\emph{frustrated Kuramoto model }(or, equivalently, \emph{frustrated rotator
model}) as introduced in \cite{sakaguchi86}:%
\begin{equation}
\theta_{i}^{\prime}(t):=\omega+\lambda\sum_{j=1}^{n}A_{i,j}\sin(\theta
_{j}(t)-\theta_{i}(t)-\alpha),\text{\qquad}t\geq0, \label{kuramoto}%
\end{equation}
where $\alpha\in\lbrack0,\pi/2)$ is a fixed, but arbitrary, phase frustration
parameter; this is chosen to be equal for every $i\in V(G)$. The parameter
$\omega$ is the natural frequency and $\lambda>0$ is the strength of the
interaction, when looking at the system as a set of oscillators. We set
$\omega=0\,$and $\lambda=1$; note that for our purposes, there is no real loss
of generality in doing so. By taking $\alpha=0$, we obtain the standard
Kuramoto model \cite{kuramoto75}. We shall consider the case $\alpha\in
(0,\pi/2)$.

Let $\theta_{i}(t)$ be a global smooth solution to the frustrated Kuramoto
model for a vertex $i\in V(G)$. There are two natural notions of phase synchronisation:

\begin{itemize}
\item Two vertices $i,j\in V(G)$ are said to exhibit \emph{phase
synchronisation} when $\theta_{i}(t)=\theta_{j}(t)$, for every $t\geq0$;

\item Two vertices $i,j\in V(G)$ are said to exhibit \emph{asymptotic phase
synchronisation} when $\lim_{t\rightarrow\infty} (\theta_{i}(t)-\theta
_{j}(t))=0$.
\end{itemize}

We study graph partitions suggested by these notions. Section 2 contains the
definition of an $\alpha$-Kuramoto partition. We show that such partitions are
a generalisation of equitable partitions. In Section 3, we exactly
characterise $\alpha$-Kuramoto bipartitions. We also illustrate our results
with several examples, revisiting and extending the mathematical analysis done
in \cite{nicosia12}. Some open problems are posed in Section 4.

\section{$\alpha$-Kuramoto partitions}

Phase synchronisation in the frustrated Kuramoto model induces a partition of
a graph. Such a partition will be called an $\alpha$-Kuramoto partition (see
the formal definition below). The $\alpha$-Kuramoto partitions are closely
related to equitable partitions, a well-known object of study in graph theory
\cite{gr04}. Theorem \ref{equi} states that every equitable partition
corresponds to phase synchronisation in the frustrated Kuramoto model.
However, not every partition induced by phase synchronisation in the model is
an equitable partition.

A \emph{partition} of a graph $G$ is a partition of $V(G)$, \emph{i.e.}, a
division of $V(G)$ into disjoint, non-empty sets that cover $V(G)$. The sets
of a partition are called \emph{parts}.

\begin{definition}
[Equitable partition]A partition $\mathcal{S}$ of a graph $G$ with parts
$S_{1},...,S_{k}$ is \emph{equitable} if the number of neighbours in $S_{j}$
of a vertex $v\in S_{i}$ depends only on the choice of the parts $S_{i}$ and
$S_{j}$. In this case, the number of neighbours in $S_{j}$ of any vertex in
$S_{i}$ is denoted $\gamma_{ij}$.
\end{definition}

The frustrated Kuramoto model as defined in (\ref{kuramoto}) is naturally
associated to a partition as follows:

\begin{definition}
[$\alpha$-Kuramoto partition]\label{kura-def} Fix $\alpha\in(0,\pi/2)$. A
partition ${\mathcal{S}}=S_{1}\cup S_{2}\cup\ldots\cup S_{k}$ of a graph $G$
is called an $\alpha$\emph{-Kuramoto partition} if there is an initial
condition
\begin{equation}
\theta_{j}(0)=x_{j},\text{\qquad}j=1,\ldots,n, \label{inl}%
\end{equation}
such that the solution to (\ref{kuramoto}) with initial condition (\ref{inl})
has the following properties:

\begin{enumerate}
\item if $i,j\in S_{l}$ for some $l$, then $\theta_{i}(t)=\theta_{j}(t)$ for
all $t\geq0$;

\item if $i\in S_{p},j\in S_{q}$ for distinct indices $p,q$, then as functions
on $[0,\infty)$, we have $\theta_{i}\neq\theta_{j}$.
\end{enumerate}
\end{definition}

\begin{example}
\label{ex1} For each $\alpha\in(0,\pi/2),$ the graph in Figure \ref{fig1} has
an $\alpha$-Kuramoto partition with two parts:\ the vertices $\{1,...,6\}$ are
in one part (black); the vertices $\{7,...,10\}$ are in the other part
(white).%
\begin{figure}[h]%
\centering
\includegraphics[
height=1.6734in,
width=2.7276in
]%
{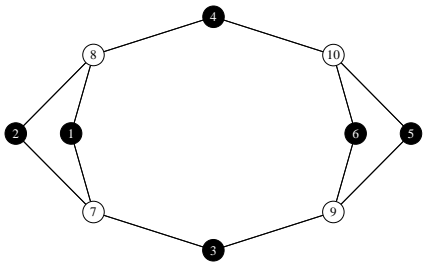}%
\caption{An $\alpha$-Kuramoto partition with two parts. }%
\label{fig1}%
\end{figure}

\end{example}

\begin{remark}
\emph{A fundamental property of the standard Kuramoto model (}$\alpha
=0$\emph{) is that }$\theta_{i}(t)=\theta_{j}(t)$\emph{, for all }$t\geq
0$\emph{\ and every two vertices }$i$\emph{\ and }$j$\emph{. It is then
crucial that }$\alpha>0$\emph{, in order to define an }$\alpha$\emph{-Kuramoto
partition.}
\end{remark}

The next result establishes a connection between equitable partitions and
$\alpha$-Kuramoto partitions. In particular, it states that every equitable
partition is also an $\alpha$-Kuramoto partition.

\begin{theorem}
\label{equi}Let $G$ be a connected graph on $n$ vertices. Suppose that
${\mathcal{S}}$ is an equitable partition of $G$. Then for any $\alpha
\in(0,\pi/2)$, ${\mathcal{S}}$ is an $\alpha$-Kuramoto partition of $G$.
\end{theorem}

\begin{proof}
Since ${\mathcal{S}}=S_{1}\cup S_{2}\cup\ldots\cup S_{k}$ is an equitable
partition, there are nonnegative integers $\gamma_{i,j}$, with $i,j=1,\ldots
,k$, such that for each pair of indices $i,j$ between $1$ and $k,$ any vertex
$p\in S_{i},$ has precisely $\gamma_{i,j}$ neighbours in $S_{j}.$ Given
scalars $c_{1}<c_{2}<\ldots<c_{k}$, we consider the following system of
differential equations
\begin{equation}
f_{i}^{\prime}(t)=\sum_{j=1}^{k}\gamma_{i,j}\sin(f_{j}(t)-f_{i}(t)-\alpha
),\text{\qquad}i=1,\ldots,k, \label{f-kur}%
\end{equation}
with initial condition $f_{j}(0)=c_{j}$, for $j=1,\ldots,k$. Observe that the
right side of (\ref{f-kur}) has continuous partial derivatives with respect to
each $f_{l}.$ Hence by Picard's existence theorem \cite{teschl2012}, there is
a (unique) solution to this equation with the given initial conditions.

Next, we consider the system (\ref{kuramoto}) with the following initial
conditions: for each $i=1,\ldots,k$ and each $p\in S_{i}$, we take $\theta
_{p}(0)=c_{i}.$ We construct a solution to this system as follows: for each
$p=1,\ldots,n$, with $p\in S_{i},$ set $\theta_{p}(t)=f_{i}(t).$ It is readily
verified that this is yields a solution to (\ref{kuramoto}) with the given
initial conditions, and again using Picard's theorem, we observe that such a
solution is unique. Evidently, for this solution, we have $\theta_{p}%
=\theta_{q}$ if $p,q\in S_{l}$, for some $l$, and $\theta_{p}\neq\theta_{q}$
if $p\in S_{l},q\in S_{m}$ and $l\neq m$.
\end{proof}

\bigskip

An \emph{automorphism} of a graph $G$ is a permutation $\pi$ of $V(G)$ with
the property that, for any two vertices $i,j\in V(G)$, we have $\{i,j\}\in
E(G)$ if and only if $\{\pi(i),\pi(j)\}\in E(G)$. Two vertices are
\emph{symmetric} if there exists an automorphism which maps one vertex to the
other. The equivalence classes consisting of symmetric vertices are called the
\emph{orbits} of the graph by $\pi$.

\begin{definition}
[Orbit partition]A partition $\mathcal{S}$ of a graph $G$ with parts
$S_{1},...,S_{k}$ is an \emph{orbit partition }if there is an automorphism of
$G$ with orbits $S_{1},...,S_{k}$.
\end{definition}

The following is an easy fact:

\begin{proposition}
\label{orbit}An orbit partition is an equitable partition. The converse is not
necessarily true.
\end{proposition}

Because of this, we have the following intuitive fact.

\begin{proposition}
Not every $\alpha$-Kuramoto partition is an orbit partition.
\end{proposition}

The proof is by counterexample. The graph in Figure \ref{fig1} does not have
an automorphism $\phi$ such that $\phi(1)=3$, even though vertices $1$ and $3$
are in the same part of an $\alpha$-Kuramoto partition.

\begin{remark}
\emph{A graph is said to be }$d$\emph{-}regular\emph{\ if the degree of each
vertex is }$d$\emph{. Phase synchronisation for a }$d$\emph{-regular graph is
special. In fact, one }$\alpha$\emph{-Kuramoto partition of a }$d$%
\emph{-regular graph is the singleton partition (}i.e.\emph{, it has only one
part). To show this, we need to prove that if }$G$\emph{\ is a }%
$d$\emph{-regular graph then }$\theta_{i}(t)=\theta_{j}(t)$\emph{, for all
}$i,j\in V(G)$\emph{. Fix }$\alpha$\emph{\ and let }$\theta_{i}(t)=-d\sin
(\alpha)t $, \emph{for every }$i\in V(G)$\emph{. It is now straightforward to
verify that }$\theta_{i}(t)$\emph{\ is a solution to the model.}
\end{remark}

\section{$\alpha$-Kuramoto bipartitions}

An \emph{equitable bipartition} (resp. $\alpha$\emph{-Kuramoto bipartition})
is an equitable partition ($\alpha$-Kuramoto partition) with exactly two
parts. The relationship between equitable bipartitions and $\alpha$-Kuramoto
bipartitions is special. We observe the main property of this relationship in
Theorem \ref{asymp}. First, we need a technical lemma whose proof is elementary.

\begin{lemma}
\label{tech}Suppose that $f(x)$ is a differentiable function on $[0,\infty)$
and that $f^{\prime}(x)$ is uniformly continuous on $[0,\infty)$. If
$f(x)\rightarrow0$ as $x\rightarrow\infty$, then $f^{\prime}(x)\rightarrow0$
as $x\rightarrow\infty$.
\end{lemma}

\begin{proof}
Suppose to the contrary that $f^{\prime}(x)$ does not approach $0$ as
$x\rightarrow\infty.$ Then there is a sequence $x_{j}$ diverging to infinity,
and $c>0$ such that $|f^{\prime}(x_{j})|\geq c$, for all $j\in\mathbb{N}$.
Without loss of generality we assume that $f^{\prime}(x_{j})\geq c$ for
$j\in\mathbb{N}$.

Since $f^{\prime}$ is uniformly continuous, there is $h>0$ such that for each
$j\in\mathbb{N}$ all $x\in\lbrack x_{j}-h,x_{j}+h]$ and $f^{\prime}(x)\geq
c/2$. Applying the mean value theorem to $f$, we see that for each
$j\in\mathbb{N}$, there is a $z\in\lbrack x_{j}-h,x_{j}+h]$ such that
\[
f(x_{j}+h)=f(x_{j})+hf^{\prime}(z)\geq f(x_{j})+\frac{ch}{2}.
\]
But then for all sufficiently large integers $j$, we have $f(x_{j}+h)\geq
ch/4$, contrary to hypothesis. The conclusion now follows.
\end{proof}

\bigskip

We are now ready to state the main result of this section.

\begin{theorem}
\label{asymp} Let $G$ be a connected graph on $n$ vertices. Suppose that
$1\leq k\leq n-1$ and define $S_{1}=\{1,\ldots,k\},$ $S_{2}=\{k+1,\ldots,n\}
$, and ${\mathcal{S}}=S_{1}\cup S_{2}$. For each $i=1,\ldots,n$, let
$\delta_{i,1}$ be the number of neighbours of vertex $i$ in $S_{1}$, and let
$\delta_{i,2}$ be the number of neighbours of vertex $i$ in $S_{2}$. Suppose
that $\alpha\in(0,\pi/2)$ and $\theta_{j}$ ($j=1,\ldots,n$) is a solution to
the frustrated Kuramoto model. Suppose further that:

\begin{itemize}
\item for all $i,j\in S_{1}$, we have $\theta_{i}(t)-\theta_{j}(t)\rightarrow
0$ as $t\rightarrow\infty$;

\item for all $p,q\in S_{2}$, we have $\theta_{p}(t)-\theta_{q}(t)\rightarrow
0$ as $t\rightarrow\infty$;

\item for some $i\in S_{1}$ and $p\in S_{2}$, $\theta_{i}(t)-\theta_{p}(t)$
does not converge to $0$ as $t\rightarrow\infty$.
\end{itemize}

Then one of the following holds:

\begin{enumerate}
\item the partition ${\mathcal{S}}$ is an equitable partition of $G$;

\item there are scalars $\mu_{1},\mu_{2}$, and $r$ such that%
\[
-\delta_{i,1}+\mu_{1}\delta_{i,2}=r\text{\qquad for all }i\in S_{1},
\]
and%
\[
-\delta_{j,2}+\mu_{2}\delta_{j,1}=r\text{\qquad for all }j\in S_{2}.
\]

\end{enumerate}
\end{theorem}

\begin{proof}
Set $\lambda_{1}(t)=\frac{1}{k}\sum_{i=1}^{k}\theta_{i}(t)$ and $\lambda
_{2}(t)=\frac{1}{n-k}\sum_{j=k+1}^{n}\theta_{j}(t)$. For each $i\in S_{1},$
let $\epsilon_{i}=\theta_{i}-\lambda_{1}$, and for each $j\in S_{2}$, let
$\epsilon_{j}=\theta_{j}-\lambda_{2}.$ From our hypotheses, for each
$p=1,\ldots,n$, $\epsilon_{p}(t)\rightarrow0$ as $t\rightarrow\infty,$ while
$\lambda_{1}(t)-\lambda_{2}(t)$ does not converge to $0$ as $t\rightarrow
\infty.$ Observe that each $\theta_{i}$ is differentiable with uniformly
continuous derivative, and hence the same is true of $\lambda_{1},\lambda
_{2},$ and each of $\epsilon_{1},\ldots,\epsilon_{n}.$ In particular, applying
Lemma \ref{tech} to each $\epsilon_{p},$ we find that as $t\rightarrow\infty,$
$\epsilon_{p}^{\prime}(t)\rightarrow0,$ for $p=1,\ldots,n.$

For each $i\in S_{1},$ we have
\[
\lambda_{1}^{\prime}+\epsilon_{i}^{\prime}=\sum_{j\sim i,j\in S_{1}}%
\sin(\epsilon_{j}-\epsilon_{i}-\alpha)+\sum_{j\sim i,j\in S_{2}}\sin
(\lambda_{2}-\lambda_{1}-\alpha+\epsilon_{j}-\epsilon_{i})
\]
(here we use $j \sim i$ to denote the fact that vertices $j$ and $i$ are
adjacent). Now, we rewrite this as
\begin{equation}
\lambda_{1}^{\prime}=-\delta_{i,1}\sin(\alpha)+\delta_{i,2}\sin(\lambda
_{2}-\lambda_{1}-\alpha)+\eta_{i}. \label{etai}%
\end{equation}
Observe that since $\epsilon_{p}(t),\epsilon_{p}^{\prime}(t)\rightarrow0$ as
$t\rightarrow\infty$, for $p=1,\ldots,n$, it follows that $\eta_{i}%
(t)\rightarrow0$ as $t\rightarrow\infty.$ Similarly, we find that for each
$j\in S_{2},$
\begin{equation}
\lambda_{2}^{\prime}=\delta_{j,1}\sin(\lambda_{1}-\lambda_{2}-\alpha
)-\delta_{j,2}\sin(\alpha)+\eta_{j}(t), \label{etaj}%
\end{equation}
and that for each $j\in S_{2}$, $\eta_{j}(t)\rightarrow0$ as $t\rightarrow
\infty.$

First, suppose that $\lambda_{1}(t)-\lambda_{2}(t)$ does not converge to a
constant as $t\rightarrow\infty.$ Fix two indices $p,q\in S_{1},$ and note
from (\ref{etai}) that
\[
(\delta_{p,1}-\delta_{q,2})\sin(\alpha)+\eta_{p}-\eta_{q}=(\delta_{q,2}%
-\delta_{p,2})\sin(\lambda_{2}-\lambda_{1}-\alpha).
\]
If $\delta_{q,2}-\delta_{p,2}\neq0,$ we find that as $t\rightarrow\infty,$ for
some $k\in\mathbb{N},$ either
\[
\lambda_{2}-\lambda_{1}\rightarrow\arcsin\left(  \frac{\delta_{q,1}%
-\delta_{p,1}}{\delta_{q,2}-\delta_{p,2}}\sin(\alpha)\right)  +\alpha+2\pi k
\]
or
\[
\lambda_{2}-\lambda_{1}\rightarrow+\alpha+(2k+1)\pi-\arcsin\left(
\frac{\delta_{q,1}-\delta_{p,1}}{\delta_{q,2}-\delta_{p,2}}\sin(\alpha
)\right)  ,
\]
both of which are contrary to our assumption. Hence it must be the case that
$\delta_{q,2}=\delta_{p,2}$ which immediately yields that $\delta_{q,1}%
=\delta_{p,1}$. A similar argument applies for any pair of indices $p,q\in
S_{2},$ and we deduce that ${\mathcal{S}}$ is an equitable partition.

Now we suppose that for some constant $\beta$, $\lambda_{1}(t)-\lambda
_{2}(t)\rightarrow\alpha+\beta$ as $t\rightarrow\infty.$ Note that then we
also have $\lambda_{1}^{\prime}(t)-\lambda_{2}^{\prime}(t)\rightarrow0$ as
$t\rightarrow\infty.$ Fix indices $p\in S_{1},q\in S_{2}.$ From (\ref{etai})
and (\ref{etaj}), we find, upon letting $t\rightarrow\infty$ that
\[
-\delta_{p,1}\sin(\alpha)+\delta_{p,2}\sin(\beta)=-\delta_{q,2}\sin
(\alpha)+\delta_{q,1}(-\sin(2\alpha+\beta)).
\]
Letting
\begin{equation}
\mu_{1}=\dfrac{\sin(\beta)}{\sin(\alpha)}\text{ and }\mu_{2}=\dfrac
{-\sin(2\alpha+\beta)}{\sin(\alpha)}, \label{mu12}%
\end{equation}
it now follows that there is a scalar $r$ such that, for all $i\in S_{1}$ and
$j\in S_{2}$,%
\[
-\delta_{i,1}+\mu_{1}\delta_{i,2}=r=-\delta_{j,2}+\mu_{2}\delta_{j,1}.
\]

\end{proof}

\begin{remark}
\emph{Suppose that condition 2 of Theorem \ref{asymp} holds. From (\ref{mu12})
in the proof of Theorem \ref{asymp}, it is straightforward to verify that
}$\mu_{1}+\mu_{2}=-2\cos(\alpha+\beta)$, \emph{so that necessarily }$|\mu
_{1}+\mu_{2}|\leq2$\emph{.\ In particular, for some }$k\in N,$\emph{\ }%
\begin{equation}
\beta=\arccos\left(  -\frac{\mu_{1}+\mu_{2}}{2}\right)  -\alpha+2\pi k.
\label{beform}%
\end{equation}
\emph{Observe also that }%
\begin{align*}
\mu_{2}  &  =\frac{-\sin(\alpha)\cos(\alpha+\beta)-\cos(\alpha)\sin
(\alpha+\beta)}{\sin(\alpha)}\\
&  =\frac{\mu_{1}+\mu_{2}}{2}-\cot(\alpha)\sin(\alpha+\beta)\\
&  =\frac{\mu_{1}+\mu_{2}}{2}-\cot(\alpha)\sqrt{1-\left(  \frac{\mu_{1}%
+\mu_{2}}{2}\right)  ^{2}}.
\end{align*}
\emph{It now follows that }%
\begin{equation}
\alpha=\arctan\left(  \frac{\sqrt{4-(\mu_{1}+\mu_{2})^{2}}}{\mu_{1}-\mu_{2}%
}\right)  . \label{alform}%
\end{equation}
\emph{Moreover, since }$\alpha\in(0,\pi/2),$\emph{\ in fact we must have
}$|\mu_{1}+\mu_{2}|<2$\emph{\ and }$\mu_{1}>\mu_{2}.$\emph{\ }
\end{remark}

\begin{remark}
\emph{Observe that if a graph }$G$\emph{\ happens to satisfy condition 2 of
Theorem \ref{asymp}, and if }$\mathcal{S}$\emph{\ is not an equitable
partition, then the parameters }$\mu_{1},\mu_{2}$\emph{\ and }$r$\emph{\ are
readily determined, as follows. Suppose without loss of generality that
}$|S_{1}|\geq2$\emph{\ and that }$\delta_{i,2}\neq\delta_{j,2}$\emph{\ for
some }$i,j\in S_{1}.$\emph{\ Then we find that }$\mu_{1}=(\delta_{j,1}%
-\delta_{i,1})/(\delta_{j,2}-\delta_{i,2})$,$\emph{\ }$\emph{from which we may
find }$r$\emph{\ easily. Since }$G$\emph{\ is connected, }$\delta_{j,1}%
$\emph{\ is nonzero for some }$j\in S_{2},$\emph{\ so that }$\mu_{2}%
$\emph{\ is then also easily determined. Since }$\mu_{1},\mu_{2}$\emph{\ are
determined by }$G$\emph{, so are }$\alpha$\emph{\ and }$\beta.$\emph{\ Suppose
now that the parameters }$\mu_{1},\mu_{2}$\emph{\ so determined are such that
}$|\mu_{1}+\mu_{2}|<2$\emph{\ and }$\mu_{1}>\mu_{2},$\emph{\ and let }%
$\alpha,\beta$\emph{\ be given by (\ref{alform}) and (\ref{beform}),
respectively. Select a scalar }$c,$\emph{\ and set initial conditions }%
$\theta_{i}(0)=c$,\emph{\ for }$i\in S_{1}$\emph{, and }$\theta_{j}%
(0)=c+\alpha+\beta$\emph{, for }$j\in S_{2}$. \emph{Then the solution to the
frustrated Kuramoto model is readily seen to be }$\theta_{i}(t)=c+r\sin
(\alpha)t$\emph{\ (}$i\in S_{1}$\emph{) and }$\theta_{j}(t)=c+r\sin
(\alpha)t+\alpha+\beta$\emph{\ (}$j\in S_{2}$\emph{). In particular, we find
that if $\mathcal{S}$ is not an equitable bipartition, but satisfies condition
2 of Theorem \ref{asymp}, then $\mathcal{S}$ is an $\alpha$--Kuramoto
bipartition for $\alpha$ given by (\ref{alform}).}
\end{remark}

\begin{example}
\textit{\label{kura_eg}}\emph{Consider the graph }$G$\emph{\ in Figure
\ref{fig2}. Let }$S_{1}=\{1,\ldots,5\}$\emph{\ and }$S_{2}=\{6,\ldots
,10\}$.\emph{\ It is readily determined that }$G$\emph{\ satisfies condition 2
of Theorem \ref{asymp}, where our parameters are }$\mu_{1}=1/2$,\emph{\ }%
$\mu_{2}=1/2$\emph{\ and }$r=0.$\emph{\ We find that }$\alpha=\pi/2$
and\emph{\ }$\beta=\pi/6+2\pi k$\emph{\ (for some }$k\in\mathbb{N}$\emph{),
and that for an initial condition of the form }$\theta_{i}(0)=c$%
\emph{\ (}$i\in S_{1}$\emph{) and }$\theta_{j}(0)=c+2\pi/3+2\pi k$%
\emph{\ (}$j\in S_{2}$\emph{) the solution to (\ref{kuramoto}) is constant for
all }$t\geq0$. \emph{\ In particular, extending Definition \ref{kura-def}
slightly to include the case that $\alpha=\pi/2,$ we see that $S_{1}\cup
S_{2}$ can be thought of as a $\pi/2$--Kuramoto bipartition of $G$. } \emph{\
\begin{figure}[h]%
\centering
\includegraphics[
height=2.5382in,
width=2.7276in
]%
{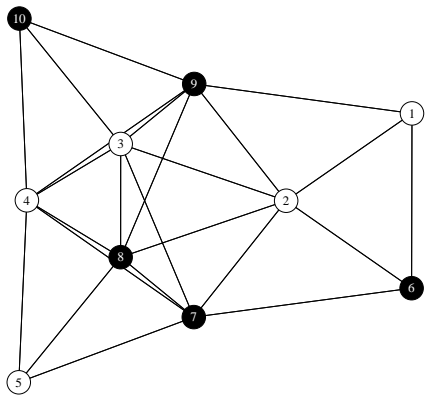}%
\caption{A graph $G$ with a partition of $V(G)$ into $\{1,\ldots,5\}$%
\emph{\ }and $\{6,\ldots,10\}$ satisfying condition 2 in Theorem \ref{asymp}.}%
\label{fig2}%
\end{figure}
}
\end{example}

\begin{example}
\label{linear}\emph{Suppose that }$p\geq4$\emph{\ is even and consider the
following adjacency matrix }%
\[
\left[
\begin{array}
[c]{c|c|c}%
0 & \mathbf{1}_{p}^{T} & \mathbf{0}_{p}^{T}\\\hline
\mathbf{1}_{p} & \mathbf{0}_{p\times p} & I_{p}\\\hline
0_{p} & I_{p} & B
\end{array}
\right]  ,
\]
\emph{where }$0_{p}$\emph{\ and }$1_{p}$\emph{\ are the zero and all ones
vector in }$\mathbb{R}^{p}$,\emph{\ respectively, }$\mathbf{0}_{p\times p}%
$\emph{\ and }$I_{p}$\emph{\ are the zero and identity matrices of order }%
$p$,\emph{\ respectively, and }$B$\emph{\ is the adjacency matrix of a graph
consisting of }$p/2$\emph{\ independent edges. Consider the partition
}$S=\{1\}\cup\{2,\ldots,2p+1\}$.\emph{\ This partition yields the following
parameters: }$\delta_{1,1}=0$\emph{, }$\delta_{1,2}=p$\emph{, }$\delta
_{i,1}=1$\emph{\ (}$i=2,\ldots,p+1$\emph{), }$\delta_{i,1}=0$\emph{\ (}%
$i=p+2,\ldots,2p+1$\emph{), }$\delta_{i,2}=1$\emph{\ (}$i=2,\ldots
,p+1$\emph{), }$\delta_{i,2}=2$\emph{\ (}$i=p+2,2p+1$\emph{). It now follows
that the graph satisfies condition 2 of Theorem \ref{asymp}, with parameters
}$\mu_{1}=-2/p$\emph{, }$\mu_{2}=-1$\emph{, and }$r=-2$.\emph{\ }

\emph{Now, take }%
\[
\alpha=\arctan\left(  \frac{\sqrt{3p^{2}-4p-4}}{p-2}\right)  \text{\emph{\ and
}}\beta=\arctan\left(  \frac{p+2}{2p}\right)  -\alpha.
\]
\emph{Setting up the initial condition}%
\begin{align*}
\theta_{1}(0)  &  =0,\\
\theta_{j}(0)  &  =\arccos\left(  \frac{p+2}{2p}\right)  ,\text{\qquad
}j=2,\ldots,2p+1,
\end{align*}
\emph{it follows that the solution to (\ref{kuramoto}) is given by }%
\begin{align*}
\theta_{1}(t)  &  =p\sin(\beta)t,\\
\theta_{j}(t)  &  =\arccos\left(  \frac{p+2}{2p}\right)  +p\sin(\beta
)t,\text{\qquad}j=2,\ldots,2p+1.
\end{align*}

\end{example}

\begin{example}
\emph{\label{latoro}Here we revisit the graph of Figure 1a) in
\cite{nicosia12}. Take }$S_{1}=\{1\}$\emph{\ and }$S_{2}=\{2,\ldots
,7\}$.\emph{\ With this partition, it is straightforward to verify that the
graph and partition satisfies condition 2 of Theorem \ref{asymp}, with
parameters }$\mu_{1}=-1/2$, $\mu_{2}=-1$\emph{, and }$r=-2$.\emph{\ It now
follows that for }$\alpha=\arctan(\sqrt{7})$\emph{\ and }$\beta=\arccos\left(
3/4\right)  $\emph{\ and initial condition }$\theta_{1}(0)=0$ \emph{and
}$\theta_{j}(0)=\alpha+\beta$\emph{\ (}$j=2,\ldots,7$\emph{), the solution to
(\ref{kuramoto}) satisfies }$\theta_{1}(t)=-2t$\emph{\ and }$\theta
_{j}(t)=\alpha+\beta-2t$\emph{\ (}$j=2,\ldots,7$\emph{). The figure is drawn
below.
\begin{figure}[h]%
\centering
\includegraphics[
height=1.2635in,
width=2.738in
]%
{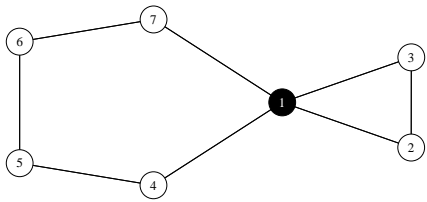}%
\caption{A graph $G$ with a partition of $V(G)$ into $\{1\}$\emph{\ }and
$\{2,\ldots,7\}$ satisfying condition 2 in Theorem \ref{asymp}. This is the
graph in Figure 1a of \cite{nicosia12}.}%
\end{figure}
}
\end{example}

\section{Open problems}

Of course the first problem to address, which is central to this paper, is
that of establishing a combinatorial characterisation of $\alpha$-Kuramoto
partitions in the general case. Further, we have just touched the surface of
asymptotic phase synchronisation in this paper, and more work needs to be done
in that direction. Exploring algorithmic applications of this notion is also a
potentially fruitful avenue for future research. We conclude the paper with a
few specific open problems.

\begin{problem}
Can we characterise the degree sequences that satisfy condition 2 of Theorem
\ref{asymp}? Specifically, suppose that $G$ is a connected graph and that
$S_{1}\cup S_{2}$ is a partition of its vertex set. Let $G_{S_{1}}$,
$G_{S_{2}}$ be the subgraphs of $G$ induced by $S_{1}$, $S_{2}$, respectively.
Find necessary and sufficient conditions on the degree sequences of $G$,
$G_{S_{1}}$ and $G_{S_{2}}$ in order that the bipartition $S_{1}\cup S_{2}$
satisfies condition 2 of Theorem \ref{asymp}. Observe that the sequences
$\delta_{i,1}$ ($i\in S_{1}$), $\delta_{j,2}$ ($j\in S_{2}$), and
$\delta_{p,1}+\delta_{p,2}$ ($p=1,\ldots,n$) must all be graphic (\emph{i.e.}
must be the degree sequence of some graph).

\end{problem}

\begin{problem}
Let $G$ be a connected graph, suppose that the bipartition $S_{1}\cup S_{2}$
satisfies condition 2 of Theorem \ref{asymp}, and let $\alpha$ be given by
(\ref{alform}). Determine the initial conditions such that the solutions of
the corresponding frustrated Kuramoto model (\ref{kuramoto}) exhibit
asymptotic synchronisation with $\lim_{t\rightarrow\infty}(\theta
_{i}(t)-\theta_{j}(t))=0$ for all $i,j\in S_{1}$ and all $i,j\in S_{2}$.
\end{problem}

\begin{problem}
Determine the values of $\alpha\in(0,\pi/2)$ such that for some connected
graph $G$, there is a bipartition of its vertex set as $S_{1}\cup S_{2}$ so
that condition 2 of Theorem \ref{asymp} holds and in addition $\alpha$ is
given by (\ref{alform}).
\end{problem}

\begin{problem}
We have seen that any regular graph admits a trivial $\alpha$-Kuramoto
partition. Adding pendant vertices or small subgraphs to a regular graph may
well change the picture. In the same spirit, it would be interesting to study
the behaviour of $\alpha$-Kuramoto partitions of graph products and other
graph operations.
\end{problem}

\emph{Acknowledgements.} We would like to thank Vito Latora and Vincenzo
Nicosia for teaching us the basic properties of the Kuramoto model and for
proposing the subject of this work. The first author is supported in part by
Science Foundation Ireland under grant number SFI/07/SK/I1216b. The second
author is supported by the Royal Society.

\end{document}